\documentclass[11pt, twoside]{article}
\usepackage[english]{babel}

\usepackage{amsmath,amsthm,amscd,latexsym, mathtools, eucal,epsfig,enumitem,bbold, color}
\usepackage{setspace}
\usepackage{commath}
\usepackage{makeidx}
\usepackage{amssymb} 
\usepackage{graphicx} 
\usepackage{fancyhdr}
\usepackage[bookmarks={true},bookmarksopen={true}]{hyperref}
\usepackage{pstricks,pstricks-add,pst-math,pst-xkey}
\usepackage{multicol}
\usepackage{pstricks,pstricks-add,pst-math,pst-xkey}
\usepackage[labelformat=simple]{subcaption}

\newtheorem{theorem}{Theorem}
\newtheorem{proposition}[theorem]{Proposition}
\newtheorem{lemma}[theorem]{Lemma}

\theoremstyle{definition}
\newtheorem{definition}[theorem]{Definition}

\theoremstyle{theorem}
\newtheorem*{AppTheo}{Main result}
\newtheorem*{remark*}{Remark}
\theoremstyle{definition}
\newtheorem*{definition*}{Definition}

\title{Invariant cone and  synchronization state stability of  the mean field models}

\author{W. Oukil$^{1}$, Ph. Thieullen$^{2}$ and A.  Kessi$^{3}$\\
\small\text{1. Department of Mathematics and Computer Science, M\'ed\'ea University,}\\
\small\text{P\^ole Urbain -26000 M\'ed\'ea, Algeria.}\\
\small\text{2. Institute of Mathematics of Bordeaux, Bordeaux University,}\\
\small{351, Cours de la Lib\'eration - F 33 405 Talence, France.}\\
\small\text{3. Faculty of Mathematics,  University of  Houari Boumediene,}\\
\small\text{ BP 32 El Alia 16111, Bab Ezzouar, Algiers, Algeria.}}

\begin{document}
\sloppy
\binoppenalty 10000
\relpenalty 10000
\date{\today}
\maketitle
\begin{abstract}
In this article we prove the stability of some mean field systems similar to the  Winfree  model in the synchronized state. The model   is governed by the coupling strength parameter $\kappa$ and the  natural frequency of each oscillator. The stability is proved independently of the number of oscillators and the distribution of the natural frequencies. The main result is proved using the positive invariant cone method for the linearized system. This method can be applied to other mean field models as in the Kuramoto model.
\end{abstract}
\section{Introduction and Main result}\label{Section:Introduction}
 In 1967 Winfree \cite{WinfreeModel}   proposed a mean field model  describing the synchronization of a population of organisms or {\it oscillators} that interact simultaneously \cite{AriaratnamStrogatz,Quinn}.  The mean field models are used, for example, in the Neurosciences to study of neuronal synchronization in the brain  \cite{Cumin2007,Ermentrouttheeffects}. The Winfree model is given by the following differential equation
\begin{equation}\label{equation:WinfreeModel}
\begin{cases}
\dot{x}_i=\omega_i-\kappa\sigma(x)R(x_{i}),\ t\geq0,\ x=(x_1,\ldots,x_n), \\
\sigma(x):=\frac{1}{n}\sum_{j=1}^{n}P(x_j),\ \forall x=(x_1,\ldots,x_n) \in\mathbb{R}^n, \\
\sup_{x \in\mathbb{R}} P(x)R(x)>0, \quad P,R \in C^2(\mathbb{R}) \ \text{are $2\pi$-periodic},
\end{cases}
\end{equation}
where $n\geq 1$  is the number of oscillators, $\sigma(x)$ is the mean field interaction, $x_i(t)$ is the phase of the i-th oscillator, and  $x(t)=(x_1(t), \ldots,x_n(t))$ is the global state of the system. We assume that the {\it natural frequencies} are chosen  indifferently in some interval about $\omega=1$,
\begin{equation}
\label{equation:ChoiceNaturalFrequencies}
\omega_i  \in  (1-\gamma,1+\gamma), \quad\text{where}\quad  \gamma \in (0,1).
\end{equation}
The {\it coupling strength} $\kappa$ is taken in the interval $(0,1)$. Let 
\begin{equation} \label{equation:aprioriBounds}
 M :=  16\max \{\|P^{(i)} \|_\infty \|R^{(j)}\|_\infty :   0\le i,j\le 2\},
\end{equation}
be a constant used explicitly in some estimates measuring the size of the mean field.

We first define the notions of  invariance and  stability. Let  $n\in\mathbb{N}^*$ and $F: \mathbb{R}^{n}\to\mathbb{R}^{n}$  be a $C^1$ vector field. Denote $DF$ its Jacobian and assume
\[
\max\{\sup\limits_{z \in \mathbb{R}^n}\norm{F(z)},\ \sup\limits_{z \in \mathbb{R}^n}\norm{DF(z)}\}<\infty.
\]
where $\norm{.}$ is the usual matrix norm. Let $\phi^{t} : \mathbb{R}^n \to \mathbb{R}^n$ be the flow of the  autonomous system
\begin{equation}\label{systemannexe}
\dot{x}=F(x),\ t\geq 0.
\end{equation}
\begin{definition}[Invariance]\label{InvariantOpenSetDefinition}
Let   $C \subset \mathbb{R}^n$ be an open set. We say that $C$ is {\it $\phi^t$-positively invariant}  for the system \eqref{systemannexe}, if $\phi^{t}(C) \subset C$ for all $t \geq 0$.
\end{definition}
\begin{definition}[Stability]\label{AnnexeCAStabAnnexe}
Let  $C \subset \mathbb{R}^n$ be an open set. We say that the system\eqref{systemannexe} is  {\it $\phi^t$-positively stable on $C$}, if  $C$ is $\phi^t$-positively invariant and
\begin{multline*}
\exists \lambda > 0,\ \forall x \in C,  \ \exists \delta>0, \ \forall y \in C:\  \\
 \norm{x-y}<\delta \ \implies \  \|\phi^{t}(x)-\phi^{t}(y) \|\le \lambda\norm{x-y}, \ \ \forall t \geq 0.
\end{multline*}
\end{definition}
\noindent  Let $\Phi^t$ be the flow of the Winfree model \eqref{equation:WinfreeModel}. The existence of a synchronization state in the Winfree model   is proved in \cite{OukilKessiThieullen}   for every number $n$ of oscillators and every choice of natural frequencies. We recall the main  synchronization hypothesis used in \cite{OukilKessiThieullen},
\[
\int_0^{2\pi}\! \frac{P(s)R'(s)}{1-\kappa P(s)R(s)} \,ds > 0,\quad \forall \kappa \in(0,\kappa_*), \tag{H}
\] 
where $\kappa_*$ is the  {\it locking bifurcation critical} parameter $\kappa_*$  defined by 
\begin{equation} \label{equation:kappaStar}
\kappa_* := \Big(\sup_{x \in \mathbb{R}} P(x)R(x) \Big)^{-1}.
\end{equation}
We proved in \cite{OukilKessiThieullen}  there exists an open set
\[
U \subset \Big\{ (\gamma,\kappa) \in (0,1) \times (0,\kappa_*) : 1-\gamma-\frac{\kappa}{\kappa_*} >0 \Big\}
\] 
containing in its closure $\{0\}\times [0,\kappa_*]$,  such that for every $n\geq1$ and every parameter $(\gamma,\kappa) \in U$  there exist two constants $D\in(0,1)$  and $\alpha(\gamma,\kappa,D)$,
\begin{gather} \label{equation:alphaofDelta}
\begin{cases}
\displaystyle 1-\gamma- M \kappa D -\kappa/\kappa_*>0, \\
\displaystyle \alpha(\gamma,\kappa,D) := \frac{2\gamma + M \kappa D^2}{1-\kappa/\kappa_*} + \frac{(2\gamma+ M\kappa D^2+ M \kappa D)(\gamma+ M \kappa D)}{(1-\gamma- M \kappa D -\kappa/\kappa_*)(1-\kappa/\kappa_*)},
\end{cases}
\end{gather}
and a $C^2$  $2\pi$-periodic  function $\Delta_{\gamma,\kappa} : \mathbb{R} \to (0,D)$  solution of
\begin{equation}\label{equation:Delta}
\begin{split}
 \frac{d}{ds}\Delta_{\gamma,\kappa}(s) = -\frac{\kappa P(s)R'(s)}{1-\kappa P(s)R(s)}\Delta_{\gamma,\kappa}(s) + \alpha(\gamma,\kappa,D), 
\end{split}
\end{equation}
and  a  $\Phi^t$-positively invariant open set $C_{\gamma,\kappa}^n$  independent of choice of the natural frequencies $(\omega_i)_{i=1}^n$, 
\begin{equation} \label{invariantwinfreeset}
C_{\gamma,\kappa}^n := \Big\{ x=(x_i)_{i=1}^n \in \mathbb{R}^n \,:\, \max\limits_{i,j}|x_j-x_i| < \Delta_{\gamma,\kappa}\Big( \frac{1}{n}\sum_{i=1}^n x_i \Big) \Big\}.
\end{equation}

The following main result asserts that $C_{\gamma,\kappa}^n $ is positively stable.

\begin{AppTheo}\label{mainresult}
Consider the Winfree model \eqref{equation:WinfreeModel} and assume that  hypothesis \text{\rm(H)} is satisfied. Then for every parameter $(\gamma,\kappa) \in U$, for every $n\geq1$ and  every choice of natural frequencies $(\omega_i)_{i=1}^n$ as in \eqref{equation:ChoiceNaturalFrequencies},  the Winfree model \eqref{equation:WinfreeModel} is $\Phi^t$-positively stable on $C_{\gamma,\kappa}^n$.
\end{AppTheo}

Using a more refined version of Theorem 2 in Saito, see \cite{Fink1974_01, Saito1971_01}, one can prove  the existence of a uniform {\it rotation vector} $\rho \in\mathbb{R}^n$ such that for every initial condition $x \in C_{\gamma,\kappa}^n$
\[
\Phi^t(x)=\rho t + p_x(t), \quad \forall t\geq0
\]
where  $p_x(t)$ is an {\it{almost periodic function}}.

\section{Invariant cone and stability}

We study in this section  the stability of a system of the form \eqref{systemannexe} using the  positive invariant cone method for the linearized equation. Propositions \ref{AnnexedernFinstab} and \ref{LemmeStabilityMeanField}  are the two main ingredients that guarantee the stability of the Winfree model. 
We actually consider more generally a parametrized linear system of the form,
\begin{equation}\label{LinearRef}
\dot{y}=A(x,t)y, \quad t\geq0, \ x \in C,
\end{equation}
where  $C$ is an open set and  $A(x,t)$  is a  continuous $n\times n$ matrix function on $C \times \mathbb{R}^+$.  Let $\Psi_x^{t}$  be the fundamental matrix of  \eqref{LinearRef} parametrized by $x\in C$.  The  fundamental matrix cocycle of the system \eqref{LinearRef} is denoted by
\[
\Psi_x^{t,t'}(z):=\Psi_x^{t}\big(\Psi_x^{t'}\big)^{-1}(z),\ \forall z\in\mathbb{R}^n,\ \forall t\geq t'\geq0.
\]
Let $V_+$ be the positive cone defined by
\begin{equation} \label{equation:positiveCone}
V_+ := \{(z_1,\ldots,z_n)\in \mathbb{R}^n :\ z_i\geq0,\ \forall i=1,\ldots, n\}.
\end{equation}
\begin{definition}\label{positifregularity}
Consider the linear system \eqref{LinearRef}. We say that the cone $V_+$ is {\it $\Psi^{t}_x$-positively invariant uniformly in $x \in C$} if  
\[
\exists \delta >0,\ \forall x\in C,\ \exists t_x\in[0,\delta]:\quad  \Psi^{t,t_x}_ {x}(V_+)\subset V_+,\quad \forall t\geq t_x.
\]
\end{definition}

\begin{definition}
Consider the linear system \eqref{LinearRef}. Let $\Psi_x^{t}$  be its  fundamental matrix. We say that \eqref{LinearRef}  is  {\it $\Psi_x^t$-positively stable uniformly in $C$} if 
\begin{align*}
&\exists \lambda > 0,\ \forall x \in C, \ \forall t \geq 0, \quad   \norm{\Psi_x^{t}}\le \lambda.
\end{align*}
\end{definition}

We study in the next proposition  the stability of some classes of nonlinear systems using the positive invariant cone method.

\begin{proposition}\label{AnnexedernFinstab}
Consider the system \eqref{systemannexe}. Let be $F:=(f_1,\ldots,f_n)$. Suppose that there exists a  $\phi^t$-positively invariant open set $C\subset\mathbb{R}^n$ and there exists $\alpha>0$ such that
\[
 f_{i}(\phi^t(x))\geq \alpha,\ \ \forall x \in C, \ \forall t\geq 0, \ \forall i \in \{1,\ldots,n\}.
\]
Let  $x\in C$ and  $\Psi_x^{t}$  be the  fundamental matrix of the linearized system
\begin{equation}\label{Annexepose}
\dot{y} = DF(\phi^{t}(x))y,\quad t\geq0.
\end{equation}
Suppose that $V_+$ as in  \eqref{equation:positiveCone} is $\Psi_ {x}^t$-positively invariant uniformly in $C$, then  \eqref{systemannexe} is  $\phi^t$-positively stable on  $ C$.
\end{proposition}
To prove  Proposition  \ref{AnnexedernFinstab} we use the next Lemma which gives a sufficient condition of the stability of the system \eqref{systemannexe} as defined in Definition  \ref{AnnexeCAStabAnnexe}.
\begin{lemma}\label{AnnexeCBstabBsystem}
Consider the system \eqref{systemannexe}. Suppose that there exists a  $\phi^t$-positively invariant open set $C\subset\mathbb{R}^n$   such that the linear system
\begin{equation}\label{LinearLemmaAnnexeC}
\dot{y}= DF(\phi^{t}(x)) y,\quad \forall t\geq 0, \ \forall x \in C,
\end{equation}
is $\Psi_x^t$-positively stable uniformly in $C$. Then  \eqref{systemannexe} is $\phi^t$-positively stable  on $C$.
\end{lemma}
\begin{proof}
The system \eqref{systemannexe} can be written as
\[
\frac{d}{dt}{D\phi^{t}(x)}=DF(\phi^{t}(x)) {D \phi^{t}(x)} \quad\text{with}\quad \Psi_x^t = D\phi^t(x).
\]
Since the  system  \eqref{LinearLemmaAnnexeC}  is $\Psi_x^t$-positively stable uniformly in $C$, we have
\begin{equation}\label{Annexejacob}
\exists \lambda>0,\ \forall x \in C,  \   \forall t \geq 0, \quad \norm{D \phi^{t}(x)} \le \lambda.
\end{equation}
Let $(z_1, z_2) \in C \times C$  such that $z(s):=(1-s)z_2+s z_1 \in C$ for all $s \in [0,1]$. Then
\begin{multline*}
\norm{\phi^{t}(z_1)-\phi^{t}(z_2)} =\norm{\int_{0}^{1}\!\frac{d}{ds}\phi^{t}(z(s))ds}=\norm{\int_{0}^{1}\! D \phi^{t}(z(s))\frac{dz(s)}{ds} \,ds}, \\
=\norm{\int_{0}^{1}\!D \phi^{t}(z(s))(z_1-z_2) \,ds}  \le\!\sup\limits_{s \in [0,1]}\norm{D \phi^{t}(z(s)) (z_1-z_2)}.
\end{multline*}
Finally, use the fact $z(s) \in C$  and use equation \eqref{Annexejacob} to obtain 
\[
\norm{\phi^{t}(z_1)-\phi^{t}(z_2)} \le \lambda\norm{z_1-z_2},   \ \ \forall t \geq 0,
\]
which implies that the  \eqref{systemannexe} is $\phi^t$-positively stable  on $C$.
\end{proof}
\begin{proof}[Proof of Proposition \ref{AnnexedernFinstab}]
Since $V_+$ is $\Psi_{x}^t$-positively invariant uniformly in $C$,
\[
\exists \delta>0,\ \forall x\in C,\ \exists t_x\in [0,\delta]:\  z\in V_+\implies \Psi^{t,t_x}_{x}(z)\in V_+,\quad\forall t\geq t_x.
\]
Let be $\eta:=\alpha^{-1}$,  $x\in C$,  $y\in \mathbb{R}^n$, and denote
\begin{align*}
z^+:=\eta\norm{y}F(\phi^{t_x}(x))+y\quad\text{and}\quad z^-:=\eta\norm{y}F(\phi^{t_x}(x))-y.
\end{align*}
On the one hand, $z^{+}:=(z^{+}_1,\ldots,z^{+}_n)\in V_+$ and $z^{-}:=(z^{-}_1,\ldots,z^{-}_n) \in V_+$,
\begin{align*}
\min\limits_{1\le i \le n}\{z^{-}_i,z^{+}_i\}&\geq \eta\norm{y}\min\limits_{1\le i \le n} \big\{\inf_{x\in C}f_i(\phi^{t_x}(x)) \big\}-\norm{y}\geq(\eta\alpha-1)\norm{y}=0.
\end{align*}
On the other hand, $F(\phi^t(x))$ is  solution of the linearized system \eqref{Annexepose} 
\begin{gather*}
\frac{d}{ds}   F(\phi^{s }(x)) = DF(\phi^s(x)) F(\phi^s(x)),  \\
F(\phi^t(x)) = \Psi_x^{t,t_*}F(\phi^{t_*}(x)).
\end{gather*}
Since $V_+$ is $\Psi_{x}^t$-positively invariant uniformly in $C$, we obtain
\begin{align*}
&\eta\norm{y}F(\phi^{t}(x))+\Psi^{t,t_x}_x( y)=\Psi^{t,t_x}_x (z^{+}) \in V_+, \ \forall t \geq t_x, \\
&\eta\norm{y}F(\phi^{t}(x))-\Psi^{t,t_x}_x(y)=\Psi^{t,t_x}_x(z^{-})\in V_+, \ \forall t \geq t_x.
\end{align*}
Put  $r:=\max\{\norm{F(\phi^{t}(x))},\norm{DF(\phi^{t}(x))}\}<+\infty$ we obtain 
\begin{gather*}
\norm{\Psi^{t,t_x}_{x}(y)}\le \eta r\norm{y}, \ \forall t\geq \delta, \quad \norm{\Psi^{t}_{x}} \le \exp(r\delta), \ \forall t\in [0,\delta], \\
\norm{\Psi^{t}_{x}(y)}\le \lambda  \norm{y},\quad\forall x\in C,\ \forall t \geq 0,
\end{gather*}
where $\lambda := \eta r\exp(r\delta)$. The linearized system \eqref{Annexepose} is $\Psi_x^t$-positively stable uniformly in $C$. Lemma  \ref{AnnexeCBstabBsystem}  implies that  \eqref{systemannexe} is $\phi^t$-positively stable  on  $C$.
\end{proof}

We give  in the following proposition a sufficient condition  for the invariance of the cone $V_+$ .

\begin{proposition}\label{LemmeStabilityMeanField}
Let $p,q : \mathbb{R} \to \mathbb{R}$ be continuous $2\pi$-periodic functions, and $g_i,h_{i,j} : [0,+\infty) \to \mathbb{R}$, $1 \leq i,j \leq n$, be continuous functions. Consider the linear non-autonomous ODE
\begin{equation}\label{LinearRefNN}
\frac{dz_i}{ds} = g_i(s) z_i + \frac{1}{n} \sum_{j=1}^n \big( p(s)  +  h_{i,j}(s) \big) z_j, \quad \forall s \geq0, \ \forall 1\leq i\leq n.
\end{equation}
Assume there exists a constant $D>0$ and a continuous $2\pi$-periodic function $\delta : \mathbb{R} \to (0,D)$ such that
\begin{itemize}
\item $\displaystyle \int_{0}^{2\pi}\!\! p(s) \,ds >0$,
\item $\displaystyle \frac{d \delta}{ds} = -p(s) \delta +q(s), \ \forall s \geq0$,
\item $\displaystyle 0 \leq g_i(s) \leq \frac{q(s)}{4D}, \ |h_{i,j}(s)| \leq \frac{q(s)}{8D}, \ \forall s\geq0$.
\end{itemize}
Let  $\Psi^{s,s'}$ be the fundamental matrix of \eqref{LinearRefNN}. Then  there exists $s_*\in[0,2\pi]$ such that  $\Psi^{s,s_*}(V_+)\subset V_+$ for all $s\geq s_*$.
\end{proposition}

\begin{proof}
 Let be $ {s}_*\in [0,2\pi]$ satisfying 
\begin{equation}\label{setoiletheta}
\max_{s\in[0,2\pi]}\delta(s)=\delta(s_*).
\end{equation}
Let  $\Psi^{s,s'}=(\Psi_{1}^{s,s'},\ldots,\Psi_{n}^{s,s'})$ be the fundamental matrix cocycle of  \eqref{LinearRefNN}. Let $\mathring{V}_+$ be the interior of the set ${V}_+$, $z_*\in \mathring{V}_+$ fixed, and $z(s) = \Psi^{s,s_*}(z_*)$.  By continuity,
\[
\exists s_1>s_*:\quad z(s)\in  \mathring{V}_+,\ \forall s\in [s_*,s_1).
\]
Define
\[
S:=\sup\big\{s>s_*: \ z(s')\in  \mathring{V}_+,\quad\forall s'\in [s_*,s)\big\}.
\]
The proposition is proved if we show $S=+\infty$. By contradiction, suppose that $S<+\infty$, then
\begin{equation}\label{provcontradrevised}
z(S) \notin  \mathring{V}_+.
\end{equation}
Define 
\[
\mu(s) := \frac{1}{n} \sum_{i=1}^n z_i(s).
\]
 By  uniqueness of solutions $\mu(s)>0, \ \forall s\in [s_*,S]$.  Then for all $s\in[s_*,S]$,
\begin{gather*}
\frac{dz_i}{ds} = g_i(s) z_i + (p(s)+h_i(s))\mu(s), 
\end{gather*}
where
\[
h_i(s) := \frac{\sum_{j=1}^n h_{i,j}(s) z_j(s)}{\sum_{j=1}^n z_j(s)}.
\]
Define
\[
g(s) := \frac{\sum_{i=1}^n g_i(s)z_i(s)}{\sum_{i=1}^n z_i(s)}, \ \ \text{and} \ \ h(s) := \frac{1}{n}\sum_{i=1}^n h_i(s).
\]
Then
\[
0 \leq  g(s) \leq \frac{q(s)}{4D}, \quad |h_i(s)| \leq  \frac{q(s)}{8D}, \quad |h(s)| \leq \frac{q(s)}{8D}.
\]
Define 
\[
a(s) := g(s) + p(s) + h(s), \quad \forall s \geq s_*.
\]
Then 
\[
\frac{d \mu}{ds} = a(s) \mu, \quad \mu(s) = \mu(s_*) \exp \Big(\int_{s_*}^s a(\zeta) \,d\zeta \Big).
\]
Since $|p(s)+h_i(s)-a(s)| = |-g(s)+h_i(s)-h(s)| \leq q(s)/(2D)$, we have
\begin{align*}
\frac{d z_i}{ds} &\geq (p(s)+h_i(s))\mu(s), \\
&\geq (p(s)+h_i(s)) \mu(s_*) \exp \Big( \int_{s_*}^s a(\zeta) \,d\zeta \Big), \\
\frac{z_i(s) -z_i(s_*)}{\mu(s_*)} &\geq \int_{s_*}^s \big(p(s')+h_i(s')-a(s') \big) \exp\Big( \int_{s_*}^{s'} a(\zeta) \,d\zeta \Big) \,ds' \\
&\quad + \int_{s_*}^s a(s')  \exp\Big( \int_{s_*}^{s'} a(\zeta) \,d\zeta \Big) \,ds' \\
&\geq \exp\Big( \int_{s_*}^{s} a(\zeta) \,d\zeta \Big) - 1 -\int_{s_*}^s \frac{q(s')}{2D} \exp\Big( \int_{s_*}^{s'} a(\zeta) \,d\zeta \Big) \,ds'.
\end{align*}
Multiplying by $\delta(s_*) \exp \big( -\int_{s_*}^s  a(\zeta) \,d\zeta \big) $ and using ${\delta}( {s}_*)<D$, we get
\begin{align*}
\frac{\delta(s_*) z_i(s)}{\mu(s)} &\geq \delta(s_*) -  \delta(s_*) \exp\Big( -\int_{s_*}^{s} a(\zeta) \,d\zeta \Big) \\
&\quad - \int_{s_*}^s \frac{q(s')}{2} \exp\Big( -\int_{s'}^{s} a(\zeta) \,d\zeta \Big) \,ds'.
\end{align*}
Let $\tilde \delta(s)$ be the unique solution of
\[
\frac{d\tilde\delta}{ds} = -a(s) \tilde \delta + \frac{q(s)}{2}, \ \forall s \in[s_*,S], \quad \tilde\delta(s_*) = \delta(s_*).
\]
Then 
\begin{align} 
\tilde \delta(s) =&   \delta(s_*) \exp\Big( -\int_{s_*}^{s} a(\zeta) \,d\zeta \Big)  \notag \\
&\quad  \int_{s_*}^s \frac{q(s')}{2} \exp\Big( -\int_{s'}^{s} a(\zeta) \,d\zeta \Big) \,ds', \notag \\
\frac{\delta(s_*) z_i(s)}{\mu(s)} &\geq \delta(s_*) -  \tilde\delta(s), \quad \forall s\in[s_*,S]. \label{equcontzfinal}
\end{align}
Notice that
\[
a(s) = g(s) + p(s) + h(s) \geq p(s) -\frac{q(s)}{2D}, \quad \forall s \in[s_*,S].
\]
Then
\[
\frac{d\tilde \delta}{ds} \le -p(s) \tilde \delta + \frac{q(s)}{2} \Big( 1 + \frac{\tilde\delta}{D} \Big), \quad \forall s \in[s_*,S].
\]
To obtain $z(S)\in  \mathring{V}_+$ and get a contradiction with  \eqref{provcontradrevised},  it is sufficient to prove that $\tilde\delta(s)\le{\delta}(s), \ \forall s\in(s_*,S]$. For that, we use  the comparison principle of differential equations.  Since $0<\tilde\delta(s_*)={\delta}( {s}_*)<D$ and
\[
\frac{d\tilde\delta}{ds}(s_*)< -p(s_*) \tilde\delta(s_*)+ q(s_*) = \frac{d \delta}{ds}(s_*)
\]
there exists $\epsilon> 0$ such that $\tilde\delta(s )<  \delta(s)$ for all $s\in (s_*,s_*+\epsilon)$. Define
\[
\tilde{S}:=\sup\big\{s \in[s_*,S] : \tilde\delta(s')<  \delta(s'),\quad\forall s'\in (s_*,s] \big\}.
\]
We show that that $\tilde{S} = S$. By contradiction, if $\tilde{S}< S$, then $\tilde\delta(\tilde{S})=\delta(\tilde{S})$, 
\[
\frac{d\tilde\delta}{ds}(\tilde S) < -p(\tilde S) \tilde \delta(\tilde S) +q(\tilde S) = \frac{d \delta}{ds}(\tilde S),
\]
and we could find  $s< \tilde{S}$ close enough to $\tilde{S}$ such that $\tilde \delta(s) >  \delta(s)$. We have obtained a contradiction.  Then $\tilde{S} = S $ and $\tilde\delta(\tilde{S})<\delta(\tilde{S})\le {\delta}(s_*)$. Equation \eqref{equcontzfinal} implies $z(S) \in \mathring{V}_+$, which is a contradiction with \eqref{provcontradrevised}. We have obtained $z(s)\in  \mathring{V}_+ $ for all $ s\geq {s}_*$. By continuity of the fundamental matrix cocycle,  we have   proved that $ z(s) \in V_+$ for all $z(s_*)\in V_+$ and all $s\geq  {s}_*$.
\end{proof}

\section{Proof of the Main result}

We prove in this Section the Main result of  Section \ref{Section:Introduction}. We consider  the Winfree model \eqref{equation:WinfreeModel} and its associated flow  $\Phi^t$.  We recall that  the  Winfree model  satisfies the hypothesis \text{(H)}.  The  {\it linearized Winfree model} is given by
\begin{gather}\label{lin}
\begin{cases}
\displaystyle \frac{dy}{dt}=D\mathcal{W}(\Phi^t(x))y,\quad t\geq0,\quad y=(y_1,\ldots,y_n), \\
\displaystyle \mathcal{W}_i(x):= \omega_i - \kappa\sigma(x)R(x_i), \quad x=(x_1,\ldots,x_n)\in\mathbb{R}^n, \\
\displaystyle \frac{\partial \mathcal{W}_i}{\partial x_j} = -\kappa \Big[ \sigma(x) R'(x_i) \delta_{i,j}  + \frac{R(x_i)P'(x_j)}{n} \Big]. 
\end{cases}
\end{gather} 

We fix $(\gamma,\kappa)\in U$ and an initial condition $x_*\in C_{\gamma,\kappa}^n$ defined in \eqref{invariantwinfreeset}. We denote by $\Psi^{t}_{x_*}$ the fundamental matrix of \eqref{lin}. Let $x(t) = \Phi^t(x_*)$ be the solution of \eqref{equation:WinfreeModel} starting at $x_*$, and
\[
\mu(t) := \frac{1}{n} \sum_{i=1}^n x_i(t), \quad \forall t\geq0.
\]
The main idea of the proof is to rewrite the linearized Winfree model by making a change of time $t \leftrightarrow s$ and a linear change of the tangent vectors $y \leftrightarrow z$. We first notice that the velocity of $\mu$ is strictly positive,
\begin{align*}
\frac{d\mu}{dt} &= \frac{1}{n} \sum_{i=1}^n \omega_i -\kappa \sigma(x) \frac{1}{n} \sum_{i=1}^n R(x_i), \\
&\geq \big( 1 - \kappa \sigma(\mu) R(\mu) \big) - \big( \gamma + \kappa M D \big) \geq 1 - {\kappa}/{\kappa_*} - \gamma - \kappa MD>0.
\end{align*}
The first inequality uses the definition of the constant $M$ in \eqref{equation:aprioriBounds}, the estimates \eqref{equation:ChoiceNaturalFrequencies} on the natural frequencies, the fact that $C_{\gamma,\kappa}^n$ is positively invariant,  that $\Delta_{\gamma,\kappa}$ defined in \eqref{equation:Delta} is bounded from above by $D$, and the simple estimate,
\begin{gather*}
|x_i - \mu| \leq  \Delta_{\gamma,\kappa}(\mu) \leq D,\quad \forall 1 \leq i\leq n, \\
\Big| \sigma(\mu)R(\mu) - \frac{1}{n} \sum_{i=1}^n \sigma(x)R(x_i) \Big| \leq  \frac{M}{n}\sum_{i=1}^n  |x_i - \mu| \leq MD.
\end{gather*}
The second inequality uses the definition of $\kappa_*$ in \eqref{equation:kappaStar} and the third inequality uses the bound from below \eqref{equation:alphaofDelta}. Let be $s_*:=\mu(0)$. The map 
\[
t \in[0,+\infty) \mapsto \mu(t) \in [s_*, +\infty) 
\]
is a smooth diffeomorphism admitting as inverse map
\[
s\in[s_*,+\infty) \mapsto \tau(s) \in [0,+\infty).
\]
Define for $t=\tau(s) \Leftrightarrow s = \mu(t)$,
\begin{align*}
v(s) &:= \frac{d \mu}{dt}(t), \\
f_i(s) &:= \frac{\kappa \sigma(x(t)) R'(x_i(t))}{v(s)}\\
f(s) &:= \max_{1 \leq i \leq n}  f_i(s), \\
z_i(s) &:= y_i(t)\exp\Big( \int_{s_*}^s f(u) \,du \Big), \\
g_i(s) &:= f(s) -f_i(s), \\
p(s) &:= -\frac{\kappa P'(s)R(s)}{1-\kappa P(s)R(s)}, \\
q(s) &:= \frac{(1-\kappa/\kappa_*)\alpha(\gamma,\kappa,D)}{1-\kappa P(s)R(s)}, \\
h_{i,j}(s) &:= -\frac{\kappa  R(x_i(t))P'(x_j(t))}{v(s)} + \frac{\kappa  P'(s)R(s)}{1-\kappa P(s)R(s)}.
\end{align*}

\begin{lemma}\label{LastLemma}
Then
\begin{enumerate}
\item \label{item:LastLemma_1} $\displaystyle \frac{dz_i}{ds} = g_i(s)z_i + \frac{1}{n}\sum_{j=1}^n \big( p(s) + h_{i,j}(s) \big) z_j, \quad \forall s \geq 0, \ \forall 1 \leq i\leq n$,
\item \label{item:LastLemma_2}  $\displaystyle \int_0^{2\pi}p(s) \,ds>0$, 
\item \label{item:LastLemma_3}  $p,q:\mathbb{R} \to \mathbb{R}$ are continuous and $2\pi$-periodic,
\item \label{item:LastLemma_4}  $\displaystyle 0 \leq g_i(s) \leq \frac{q(s)}{4D}, \ |h_{i,j}(s)|  \leq \frac{q(s)}{8D}$.
\end{enumerate}

\end{lemma}

\begin{proof}  Using the change of variable $\tilde z_i(s) = y_i \circ \tau(s)$, $\tilde x_i(s) = x_i \circ \tau(s)$, equation \eqref{lin} becomes, $v(s) := \frac{d}{dt} \mu \circ \tau(s)$,
\begin{align*}
\frac{d\tilde z_i}{ds}(s) &= \frac{1}{v(s)} \frac{dy_i}{dt}(t) = - \frac{\kappa}{v(s)} \Big( \sigma(\tilde x)R'(\tilde x_i) \tilde z_i + \frac{1}{n}\sum_{j=1}^n R(\tilde x_i)P'(\tilde x_j)\tilde z_j \Big), \\
&= -f_i(s)\tilde z_i + \frac{1}{n}\sum_{j=1}^n \big( p(s) +h_{i,j}(s) \big)\tilde{z}_j
\end{align*}
Making the scaling $z_i(s) := \tilde z_i(s) \exp( \int_{s_*}^s f(u) \,du)$, one obtains item \ref{item:LastLemma_1}. Item \ref{item:LastLemma_2} is a consequence of hypothesis (H) and
\begin{gather*}
\frac{d}{ds} \log \Big( \frac{1}{1-\kappa P(s)R(s)} \Big) = p(s) - \frac{\kappa P(s) R'(s)}{1-\kappa P(s)R(s)}, \\
\int_0^{2\pi}\!\! p(s) \,ds = \int_0^{2\pi} \frac{\kappa P(s) R'(s)}{1-\kappa P(s)R(s)} \,ds >0.
\end{gather*}
Item \ref{item:LastLemma_3} is true by definition of $p$ and $q$. Using $|\tilde x_i(s)-s| \leq D$, the estimate on $h_{i,j}$ is given by
\begin{align*}
|h_{i,j}(s)| &\leq \frac{\kappa|R(s)P'(s)-R(\tilde x_i)P'(\tilde x_i)|}{1-\kappa P(s)R(s)} \\ 
& \hspace{3cm}+ \frac{\kappa |R(\tilde x_i)| \,|P'(\tilde x_i)|\, | v(s)-(1-\kappa P(s)R(s))|}{v(s)(1-\kappa P(s)R(s))} \\
&\leq \frac{\kappa M}{8(1-\kappa P(s)R(s))} \Big[D+ \frac{\gamma+\kappa M D}{1-\kappa/\kappa_*-\gamma-\kappa MD} \Big] \\
&\leq \frac{\alpha(\gamma,\kappa,D)}{8D} \frac{1-\kappa/\kappa_*}{1-\kappa P(s)R(s)} = \frac{q(s)}{8D}.
\end{align*}
The estimate on $g_i$ is given by
\begin{align*}
g_i(s) &\leq \max_{1\leq i,j\leq n} \frac{\kappa |\sigma(\tilde x)|\,|R'(\tilde x_i) -R'(\tilde x_j)|}{v(s)} \\
&\leq \frac{\kappa MD}{4 (1-\kappa P(s)R(s))} +\frac{\kappa MD(1-\kappa P(s)R(s)-v(s))}{4v(s)(1-\kappa P(s)R(s))} \\
&\leq \frac{\kappa MD}{4(1-\kappa P(s)R(s))}\Big[ 1 + \frac{\gamma+\kappa MD}{1-\kappa/\kappa_*-\gamma -\kappa MD} \Big] \leq \frac{q(s)}{4D}. \qedhere
\end{align*}
\end{proof}

We now conclude the proof of the main results: we will show that $V_+$ is $\Psi^t_x$-positively invariant uniformly in $C_{\gamma,\kappa}^n$; proposition \ref{AnnexedernFinstab} will imply that the Winfree model is $\Phi^t$-positively stable uniformly on $C_{\gamma,\kappa}^n$. 

The fact that $V_+$ is positively invariant is a direct consequence of proposition \ref{LemmeStabilityMeanField} applied to the linearized Winfree model  written in terms of the new variables $z(s) =(z_1(s),\ldots,z_n(s))$. Part of the hypotheses of proposition \ref{LemmeStabilityMeanField} have been proved in lemma \ref{LastLemma}. We prove in the following lemma the remaining hypothesis.

\begin{lemma}\label{LastLemma2}
There exists a continuous $2\pi$-periodic function $\delta : \mathbb{R} \to (0,D)$ such that
\[
\frac{d\delta}{ds} = -p(s) \delta +q(s), \quad \forall s\geq0.
\] 
\end{lemma}

\begin{proof}
Let ${\Delta}_{\gamma,\kappa}(s)$ be the positive $2\pi-$periodic function as in  \eqref{equation:Delta}. Define
\[
{\delta}(s):= \frac{1-\kappa/\kappa_*}{1-\kappa P(  s)R( s)} {{\Delta}_{\gamma,\kappa}( s)}.
\]
Then $\delta \leq  \Delta_{\gamma,\kappa} < D$, and
\begin{align*}
\frac{d\delta}{ds}&= \frac{(1-\kappa/\kappa_*) \kappa(P'(s)R(s)+P(s)R'(s))}{(1-\kappa P(s)R(s))^2}\Delta_{\gamma,\kappa} + \frac{1-\kappa/\kappa_*}{1-\kappa P(s)R(s)} \frac{d \Delta_{\gamma,\kappa}}{ds}, \\
&= \frac{\kappa P'(s)R(s)}{1-\kappa P(s)R(s)}\delta + \frac{(1-\kappa/\kappa_*)\alpha(\gamma,\kappa,D)}{1-\kappa P(s)R(s)} = -p(s)\delta +q(s). \qedhere
\end{align*} 
\end{proof}

\section{Conclusion}
We  studied the stability of the Winfree model in its synchronized state. The proof is based on the positive invariant cone method. The main synchronization hypothesis used in \cite{OukilKessiThieullen} is again a critical hypothesis for the linear stability.

\section*{Acknowledgements}

We would like to thank the referee for his/her precise remarks and corrections that helped us to improve the final text.



\end{document}